\documentclass{my_m2an}
\usepackage{amsmath,amsfonts,amssymb,latexsym,epic,eepic}
\newcommand{\edge}{\sigma}
\newcommand{\edges}{{\mathcal E}}
\newcommand{\edgesint}{{\mathcal E}_{{\rm int}}}
\newcommand{\edgesext}{{\mathcal E}_{{\rm ext}}}
\newcommand{\mesh}{{\mathcal M}}
\def\xC{{\mathrm C}}
\def\xL{{\mathrm L}}

\def\xBV{{\mathrm BV}}
\def\xR{{\mathbb R}}
\newcommand{\bfn}{{\boldsymbol n}}
\newcommand{\bfu}{{\boldsymbol u}}
\newcommand{\bfx}{{\boldsymbol x}}
\newcommand{\gradi}{{\boldsymbol \nabla}}
\newcommand{\dive}{{\rm div}}
\newcommand{\norm}[1]{{\lVert #1 \rVert}}
\newcommand{\normtbv}[1]{{\lVert #1 \rVert_{\mathcal{T},t,\xBV}}}
\newcommand{\normxbv}[1]{{\lVert #1 \rVert_{\mathcal{T},x,\xBV}}}
\newcommand{\dx}{\,{\rm d}\bfx}
\newcommand{\dt}{\,{\rm d}t}
\newcommand{\li}{|\hspace{-0.12em}[}
\newcommand{\ri}{]\hspace{-0.12em}|}
\newtheorem{theorem}{Theorem}[section]
\newtheorem{lemma}[theorem]{Lemma}

\theoremstyle{definition}
\newtheorem{definition}[theorem]{Definition}
\theoremstyle{proposition}
\newtheorem{proposition}[theorem]{Proposition}

\theoremstyle{remark}
\newtheorem{remark}[theorem]{Remark}

\numberwithin{equation}{section}

\begin{document}
\title{Entropy estimates for a class of schemes \\ for the  Euler equations}
\author{T. Gallou\"et}
\address{Aix-Marseille Universit\'e, CNRS,   Centrale  Marseille,  I2M, UMR 7373, 13453 Marseille, France, \\ {thierry.gallouet@univ-amu.fr}}
\author{R. Herbin}
\address{Aix-Marseille Universit\'e, CNRS,   Centrale  Marseille,  I2M, UMR 7373, 13453 Marseille, France, \\ {raphaele.herbin@univ-amu.fr}}
\author{J.-C. Latch\'e}
\address{Institut de Radioprotection et de S\^{u}ret\'{e} Nucl\'{e}aire (IRSN), Saint-Paul-lez-Durance, 13115 France \\ (jean-claude.latche@irsn.fr)}
\author{N. Therme}
\address{Institut de Radioprotection et de S\^{u}ret\'{e} Nucl\'{e}aire (IRSN), Saint-Paul-lez-Durance, 13115 France \\ (nicolas.therme@univ-nantes.fr)}
\begin{abstract}
In this paper, we derive entropy estimates for a class of schemes for the Euler equations which present the following features: they are based on the internal energy equation (eventually with a positive corrective term at the righ-hand-side so as to ensure consistency) and the possible upwinding is performed with respect to the material velocity only.
The implicit-in-time first-order upwind scheme satisfies a local entropy inequality.
A generalization of the convection term is then introduced, which allows to limit the scheme diffusion while ensuring a weaker property: the entropy inequality is satisfied up to a remainder term which is shown to tend to zero with the space and time steps, if the discrete solution is controlled in $\xL^\infty$ and BV norms.
The explicit upwind variant also satisfies such a weaker property, at the price of an estimate for the velocity which could be derived from the introduction of a new stabilization term in the momentum balance.
Still for the explicit scheme, with the above-mentioned generalization of the convection operator, the same result only holds if the ratio of the time to the space step tends to zero.
\end{abstract}
\subjclass{35Q31,65N12,76M10,76M12}
\keywords{finite volumes, staggered, Euler equations, entropy, compressible flows, analysis.}
\date{April 2017}
\maketitle
%
%
%
\section{Introduction}\label{sec:int}
The Euler system of equations is a hyperbolic system of  conservation laws (mass, momentum and energy) describing the motion of a compressible fluid. 
In the case of regular solutions, an additional conservation law can be written for an additional quantity called entropy; however, in the presence of shock waves, the (mathematical) entropy  decreases. 
It is now known that weak solutions of the Euler system satisfying an entropy inequality may be non unique \cite{chi-15-wea}; nevertheless,  entropy inequalities play an important role in providing global stability estimates. 

In order to compute approximate solutions of the Euler equations, it is natural to seek numerical schemes which satisfy some entropy inequalities; these inequalities should enable to prove that, as the mesh and time steps tend to 0, the limit of the approximate solutions, if it exists, satisfies an entropy inequality. 
A classical way of doing so is to design so-called ``entropy stable schemes" \cite{tad-16-ent}. 
Discrete entropy inequalities are known for the one dimensional case for the Godunov scheme \cite{god-59-dif} and have been derived for Roe-type schemes in the one space dimension case \cite{ism-09-aff}.
Entropy stability has also been proven in the multi-dimensional case for semi-discrete schemes on unstructured meshes \cite{mar-12-ent,ray-16-ent}.
However it is not always possible to obtain entropy stability, especially for fully discrete schemes; here we are interested in proving more general discrete entropy inequalities or estimates for some fully discrete numerical schemes for the system of the Euler equations for a perfect gas, in the multi-dimensional case; these inequalities allow to fulfill our goal, namely to show that the possible limits of the approximation satisfy an entropy inequality. 
Such a technique was used for the convergence study of a time implicit mixed finite volume--finite element scheme for the Euler-Fourier equations with a special equation of state \cite{fei-16-con} allowing to obtain a priori estimates.


Our aim here is to derive such discrete entropy inequalities for a class of fully discrete, time explicit or implicit, multidimensional schemes introduced for the Euler and Navier-Stokes equations in  \cite{her-14-ons,gra-16-unc,her-17-con,her-17-cel,gas-17-mus,gou-17-sta} (some of them are implemented in the open-source CALIF$^3$S software \cite{califs} developed at the French Institut de Radioprotection et de S\^uret\'e Nucl\'eaire); these schemes share the characteristic of solving the internal energy balance, with positive correction terms to ensure the consistency.
Such a class of schemes is referred to in the litterature as "flux splitting scheme", since it may be obtained by splitting the system by a two-step technique (usually into a "convective" and "acoustic" part), apply a standard scheme to each part (which, for the convection system, indeed yields, at first order, an upwinding with respect to the material velocity) and then sum both steps to obtain the final flux.
Works in this direction may be found in \cite{ste-81-flu, lio-93-new, zha-93-num, lio-06-seq, tor-12-flu}, and we hope that the discussion presented in this paper may be extended in some way to these numerical methods.

Let us  recall the derivation of an entropy for the continuous Euler system.
Consider the following system:
\begin{subequations}\label{eq:cont}
\begin{align}\label{eq:mass} &
\partial_t \rho + \dive( \rho\, \bfu) = 0,
\\[1ex] \label{eq:e_int} &
\partial_t (\rho\, e) + \dive(\rho \, e \, \bfu) + p\, \dive (\bfu)\geq 0,
\\ \label{eq:etat} &
 p=(\gamma-1)\, \rho\, e,
\end{align}\end{subequations}
where $t$ stands for the time, $\rho$, $\bfu$, $p$ and $e$ are the density, velocity, pressure and internal energy respectively, and $\gamma > 1$ is a coefficient specific to the considered fluid.
 System \eqref{eq:cont} may be derived from the Euler equations, subtracting the kinetic energy balance from the total energy balance.
The first equation is the mass balance, the second one is the internal energy balance, and the third one is the equation of state for a perfect gas.
It is complemented by initial conditions for $\rho$ and $e$, denoted by $\rho_0$ and $e_0$ respectively, with $\rho_0 >0$ and $e_0>0$, and by a boundary condition which we suppose to be $\bfu \cdot \bfn=0$ at any time and {\em a.e.} on $\partial\Omega$, where $\bfn$ stands for the normal vector to the boundary.

\medskip
Starting from these equations, we seek an entropy function $\eta$ satisfying:
\begin{equation}\label{eq:entropy}
\partial_t \eta(\rho,e) + \dive\bigl[ \eta(\rho,e)\, \bfu \bigr] \leq 0.
\end{equation}
To this end, we introduce the functions $\varphi_\rho$ and $\varphi_e$ be defined as follows:
\begin{equation}\label{eq:vphis}
\varphi_\rho(z)=z \log(z),\quad \varphi_e(z)=\frac{-1}{\gamma -1} \log(z), \quad \mbox{for } z >0,\end{equation}
 and show that, formally, the function $\eta$ defined by
 \begin{equation}\label{eq:def_eta}
\eta(\rho,e)=\varphi_\rho(\rho)+\rho \varphi_e(e).
\end{equation}
satisfies \eqref{eq:entropy}.

Indeed, multiplying \eqref{eq:mass} by $\varphi'_\rho(\rho)$, a formal computation yields:
\begin{equation}\label{eq:ent_m}
\partial_t \bigl[\varphi_\rho(\rho)\bigr] + \dive\bigl[ \varphi_\rho(\rho)\, \bfu \bigr]
+ \bigl[\rho\varphi'_\rho(\rho)-\varphi_\rho(\rho) \bigr] \dive(\bfu) = 0.
\end{equation}
Then, multiplying \eqref{eq:e_int} by $\varphi'_e(e)$ yields, once again formally, since $\varphi'_e(z) <0$ for $z >0$:
\begin{equation}\label{eq:ent_e}
\partial_t \bigl[\rho\, \varphi_e(e) \bigr] + \dive \bigl[ \rho \, \varphi_e(e)\, \bfu \bigr] + \varphi'_e(e)\, p \,\dive (\bfu) \leq 0.
\end{equation}
Summing \eqref{eq:ent_m} and \eqref{eq:ent_e} and noting that $\varphi_\rho$ and $\varphi_e$  have chosen such that
\begin{equation}
	\rho\varphi'_\rho(\rho)-\varphi_\rho(\rho)+ \varphi'_e(e)\,p=0,
	\label{prop-entropie}
\end{equation}
we obtain \eqref{eq:entropy}, which is an entropy balance for the Euler equations, for the specific entropy defined by \eqref{eq:def_eta}.

In this paper,  we derive some analogous discrete entropy inequalities (with a possible remainder tending to 0) for the fully discrete, time implicit or explicit  schemes of \cite{her-14-ons,gra-16-unc,her-17-con,her-17-cel,gas-17-mus,gou-17-sta}, with a possible upwinding limited to that of the convection terms with respect to the material velocity.
Note that the entropy estimates that we obtain here apply  both the staggered schemes \cite{her-14-ons,gra-16-unc,gas-17-mus} and to the colocated scheme of \cite{her-17-cel}; indeed in both cases the mass and internal energy are written on the same cells.

Depending on the time and space discretization, we obtain three types of results:
\begin{list}{-}{\itemsep=1ex \topsep=1ex \leftmargin=1.cm \labelwidth=0.3cm \labelsep=0.5cm \itemindent=0.cm}
\item local entropy estimates, {\it i.e.} discrete analogues of \eqref{eq:entropy}, in which case the scheme is entropy stable, 
\item global entropy estimates, {\it i.e.} discrete analogues of:
\begin{equation}
\frac d {dt} \int_\Omega \eta(\rho,e) \dx \leq 0. \label{eq:global-entropy}
\end{equation}
(such a relation is a stability property of the scheme; this kind of relation was proven in e.g. \cite{coq-06-sec} for a higher order scheme for the 1D Euler equations),
\item ``weak local" entropy estimates, {\it i.e.} results of the form:
\[
\partial_t \eta(\rho,e) + \dive\bigl[ \eta(\rho,e)\, \bfu \bigr] +\mathcal{R} \leq 0,  
\]
with $\mathcal{R}$ tending to zero with respect to the space and time discretization steps (or combination of both parameters), provided that the solution is controlled in reasonable norms, here, $L^\infty$ and BV norms.
Such an inequality readily yieds a "Lax-consistency" property, of the form: the limit of a convergent sequence of solutions, bounded in suitable norms, satisfies the following weak entropy inequality:
\begin{multline*}
-\int_0^T \int_\Omega  \eta(\rho,e)\, \partial_t \varphi + \eta(\rho,e)\, \bfu \cdot \gradi \varphi \dx \dt
- \int_\Omega  \eta(\rho,e)(\bfx,0)\ \varphi(\bfx,0) \dx \leq 0,
\\
\mbox{for any function }\varphi \in \xC^\infty_c \bigl([0,T)\times \bar\Omega\bigr), \varphi \geq 0.
\end{multline*}
The precise statement of this result for the explicit, implicit and semi-implicit schemes of \cite{her-14-ons,gra-16-unc,her-17-cel,gas-17-mus,gou-17-sta} may be found in \cite{her-17-cons}.
\end{list}

\medskip
This paper is organized as follows.
We first address implicit schemes (Section \ref{sec:implicit}), then explicit schemes (Section \ref{sec:explicit}).
In these two sections, the exposition is not structured in the same way: for implicit schemes, we first consider an upwind discretization for which we get a local discrete entropy inequality (Theorem \ref{thrm:impl_upw}),  and then a variant with reduced numerical diffusion, for which we  only get a global entropy estimate and a weak local entropy inequality (Theorem \ref{thrm:impl_muscl}). 
The case of explicit schemes is a little more tricky: we again consider the same two discretizations ({\it i.e.} upwind and reduced diffusion)  but we first deal with the mass balance equation, then with the internal energy equation, and combine the results to address entropy inequalities.

\medskip
Note that the term "implicit scheme" refers to a scheme where the discrete analogue of System \eqref{eq:cont} is solved by an implicit ({\it i.e.} backward Euler) time discretization; this does not prevent from any fractional step technique that would the momentum balance equation separately.
For instance, the pressure correction algorithms introduced in \cite{her-14-ons,gra-16-unc,her-17-cel} satisfy this property.
%
%
\section{Implicit schemes}\label{sec:implicit}

Let $\mesh$ be a decomposition of the domain $\Omega$, supposed to be regular in the usual sense of the finite element literature (see \textit{e.g.} \cite{cia-91-bas}).
By $\edges$ and $\edges(K)$ we denote the set of all $(d-1)$-faces $\edge$ of the mesh and of the cell $K \in \mesh$ respectively, and we suppose that the number of the faces of a cell is bounded.
The set of faces included in $\Omega$ (resp. in the boundary $\partial \Omega$) is denoted by $\edgesint$ (resp. $\edgesext$); a face $\edge \in \edgesint$ separating the cells $K$ and $L$ is denoted by $\edge=K|L$.
For $K \in \mesh$ and $\edge \in \edges$, we denote by $|K|$ the measure of $K$ and by $|\edge|$ the $(d-1)$-measure of the face $\edge$.
Let $(t_n)_{0\leq n \leq N}$, with $0=t_0 < t_1 <\ldots < t_N=T$, define a partition of the time interval $(0,T)$, which we suppose uniform for the sake of simplicity, and let $\delta t=t_{n+1}-t_n$ for $0 \leq n \leq N-1$ be the (constant) time step.

\medskip
The discrete pressure, density and the internal energy unknowns are associated with the cells of the mesh $\mesh$; they are denoted by:
\[
\big\{ p^n_K,\ \rho^n_K,\ e^n_K,\ K \in \mesh,\ 0 \leq n \leq N \big\}.
\]

\medskip
The general form of the discrete analogue of System \eqref{eq:cont} reads:
\begin{subequations}\label{eq:impl}
\begin{align}
\nonumber
& \mbox{For } K \in \mesh,\ 0 \leq n \leq N-1, 
\\[1ex] 
\label{eq:mass_i} &
\frac{|K|}{\delta t} (\rho_K^{n+1}-\rho_K^n) + \sum_{\edge\in \edges(K)} F_{K,\edge}^{n+1} = 0,
\\[1ex] \label{eq:e_int_i} &
\frac{|K|}{\delta t} (\rho_K^{n+1}e_K^{n+1}-\rho_K^n e_K^n) + \sum_{\edge\in \edges(K)} F_{K,\edge}^{n+1} e_\edge^{n+1}
+ p_K^{n+1} \sum_{\edge\in \edges(K)} |\edge|\,u_{K,\edge}^{n+1} \geq 0,
\\[1ex] \label{eq:etat_i} &
p_K^{n+1}=(\gamma-1)\, \rho_K^{n+1}\, e_K^{n+1},
\end{align}\end{subequations}
where $F_{K,\edge}^{n+1}$ is the mass flux through the face $\edge$, $e_\edge^{n+1}$ is an approximation of the internal energy at the face $\edge$, and $u_{K,\edge}^{n+1}$ stands for an approximation of the normal velocity to the face $\edge$.
Consistently with the boundary conditions, $u_{K,\edge}^{n+1}$ vanishes on every external face.
The mass flux $F_{K,\edge}^{n+1}$ reads:
\begin{equation}
F_{K,\edge}^{n+1}=|\edge|\ \rho_\edge^{n+1} u_{K,\edge}^{n+1},
\label{mass-flux}
\end{equation}
where $\rho_\edge^{n+1}$ stands for an approximation of the density on $\edge$.
Throughout the paper, we suppose that $\rho_K^n$, $e_K^n$, $\rho_\edge^n$ and $e_\edge^n$ are positive, for any $K\in\mesh$, $\edge\in\edgesint$, $0\leq n \leq N$, which is verified by the solutions of the schemes presented in \cite{her-14-ons,her-17-con,gas-17-mus,gou-17-sta} (of course, with positive initial conditions for $\rho$ and $e$).

\medskip
We recall the following two lemmas, which where proven in \cite{her-14-ons}.
They state discrete analogues of \eqref{eq:ent_m} and \eqref{eq:ent_e} respectively.
In their formulation, and throughout the paper, $\li a,\ b \ri$ stands for $[\min(a,b),\ \max(a,b)]$, for any real numbers $a$ and $b$.

\begin{lemma}\label{lem:mass}
Let $K \in \mesh$, $n$ be such that $0\leq n \leq N-1$ and let us suppose that \eqref{eq:mass_i} is verified.
Let $\varphi$ be a twice continuously differentiable function defined over $(0,+\infty)$.
Then we have:
\begin{multline*}
\frac{|K|}{\delta t} \Bigl[\varphi(\rho_K^{n+1})-\varphi(\rho_K^n)\Bigr]
+ \sum_{\edge\in \edges(K)} |\edge|\ \varphi(\rho_\edge^{n+1})\, u_{K,\edge}^{n+1} \\
+ \Bigl[\rho_K^{n+1} \varphi'(\rho_K^{n+1}) - \varphi(\rho_K^{n+1}) \Bigr] \sum_{\edge\in \edges(K)} |\edge|\  u_{K,\edge}^{n+1}
+|K|\, (R_m)_K^{n+1}= 0,
\end{multline*}
with:
\begin{multline} \label{eq:R_mass}\hspace{10ex}
|K|\,(R_m)_K^{n+1}= \frac 1 2 \frac{|K|}{\delta t}\ \varphi''(\rho_K^{n+1/2})\ (\rho_K^{n+1}-\rho_K^n)^2
\\
+ \sum_{\edge\in \edges(K)} |\edge|
\ \Bigl[\varphi(\rho_K^{n+1}) - \varphi(\rho_\edge^{n+1}) + \varphi'(\rho_K^{n+1}) (\rho_\edge^{n+1}-\rho_K^{n+1})\Bigr] u_{K,\edge}^{n+1},
\hspace{10ex}\end{multline}
where $\rho_K^{n+1/2}\in \li \rho_K^n ,\rho_K^{n+1} \ri$.
\end{lemma}

\begin{lemma}\label{lem:energy}
Let $K \in \mesh$ and $n$ be such that $0\leq n \leq N-1$.
Let $\varphi$ be a twice continuously differentiable function defined over $(0,+\infty)$.
Then:
\begin{multline*}\hspace{10ex}
\varphi'(e_K^{n+1})\ \Bigl[ \frac{|K|}{\delta t} (\rho_K^{n+1}e_K^{n+1}-\rho_K^n e_K^n) + \sum_{\edge\in \edges(K)} F_{K,\edge}^{n+1} e_\edge^{n+1} \Bigr]
\\
=\frac{|K|}{\delta t} \Bigl[\rho_K^{n+1} \varphi(e_K^{n+1})-\rho_K^n \varphi(e_K^n) \Bigr]
+ \sum_{\edge\in \edges(K)} F_{K,\edge}^{n+1}\, \varphi(e_\edge^{n+1}) + |K|\,(R_e)_K^{n+1},
\hspace{10ex}\end{multline*}
with:
\begin{multline} \label{eq:R_e}
|K|\,(R_e)_K^{n+1}=\frac 1 2 \frac{|K|}{\delta t} \rho^n_K\ \varphi''(e_K^{n+1/2})(e_K^{n+1}-e_K^n)^2
\\+\sum_{\edge\in \edges(K)} F_{K,\edge}^{n+1}\ \Bigl[\varphi(e_K^{n+1}) - \varphi(e_\edge^{n+1}) + \varphi'(e_K^{n+1}) (e_\edge^{n+1}-e_K^{n+1}) \Bigr],
\end{multline}
where $e_K^{n+1/2} \in \li e_K^n, e_K^{n+1} \ri$.
\end{lemma}
%
%
\subsection{Upwind schemes}

In this section, we suppose that the convection fluxes are approximated with a first order upwind scheme, {\it i.e.}, for $\edge \in \edgesint$, $\edge=K|L$, $\rho_\edge^{n+1}=\rho_K^{n+1}$ and $e_\edge^{n+1}=e_K^{n+1}$ if $u_{K,\edge} \geq 0$, $\rho_\edge^{n+1}=\rho_L^{n+1}$ and $e_\edge^{n+1}=e_L^{n+1}$ otherwise.
For $\edge \in \edgesext$, thanks to the boundary conditions, the convection fluxes vanish.
Let us consider the term associated to the faces in the expression \eqref{eq:R_mass} of the remainder term $(R_m)_K^{n+1}$.
We have, for any twice continuously differentiable function $\varphi$, any internal face $\edge=K|L$ and $0 \leq n \leq N-1$:
\begin{align*}
(T_m)_{K,\edge}^{n+1}& =\Bigl[\varphi(\rho_K^{n+1}) - \varphi(\rho_\edge^{n+1}) + \varphi'(\rho_K^{n+1}) (\rho_\edge^{n+1}-\rho_K^{n+1})\Bigr] u_{K,\edge}^{n+1}\\
& =
-\frac 1 2 \varphi''(\rho_{\edge,K}^{n+1})\ (\rho_\edge^{n+1}-\rho_K^{n+1})^2 u_{K,\edge}^{n+1},
\end{align*}
where $\rho_{\edge,K}^{n+1} \in \li \rho_\edge^{n+1}, \rho_K^{n+1} \ri$.
With the upwind choice, if $u_{K,\edge}^{n+1} \geq 0$, $\rho_\edge^{n+1}=\rho_K^{n+1}$ and $(T_m)_{K,\edge}^{n+1}$ vanishes.
If $u_{K,\edge}^{n+1} < 0$ and $\varphi''$ is a non-negative function ({\it i.e.} $\varphi$ is convex), $(T_m)_{K,\edge}^{n+1}$ is non-negative and so is $(R_m)_K^{n+1}$, for any $K \in \mesh$.
Since $\varphi_\rho$ defined by \eqref{eq:vphis} is indeed convex, we thus have, applying Lemma \ref{lem:mass}, that any solution to Equation \eqref{eq:mass_i} of the scheme satisfies, for $K \in \mesh$ and $0 \leq n \leq N-1$:
\begin{multline}\label{eq:ent_mi}
\frac{|K|}{\delta t} \Bigl[\varphi_\rho(\rho_K^{n+1})-\varphi_\rho(\rho_K^n)\Bigr]
+ \sum_{\edge\in \edges(K)} |\edge|\ \varphi_\rho(\rho_\edge^{n+1}) u_{K,\edge}^{n+1}
\\ + \Bigl[\rho_K^{n+1} \varphi_\rho'(\rho_K^{n+1}) - \varphi_\rho(\rho_K^{n+1} \Bigr] \sum_{\edge\in \edges(K)} |\edge|\  u_{K,\edge}^{n+1}
\leq 0.
\end{multline}
By the same arguments, we get that $(R_e)_K^{n+1} \geq 0$ for any regular convex function $\varphi$, for any $K\in\mesh$ and $0 \leq n \leq N-1$.
Hence, since $\varphi_e$ defined by Equation \eqref{eq:vphis} is convex, we get that any solution to \eqref{eq:e_int_i} satisfies:
\begin{multline}\label{eq:ent_ei}
\frac{|K|}{\delta t} \Bigl[\rho_K^{n+1} \varphi_e(e_K^{n+1})-\rho_K^n \varphi_e(e_K^n) \Bigr]
+ \sum_{\edge\in \edges(K)} F_{K,\edge}^{n+1}\, \varphi_e(e_\edge^{n+1})
\\+ \varphi'_e(e_K^{n+1})\,p_K^{n+1} \sum_{\edge\in \edges(K)} |\edge|\  u_{K,\edge}^{n+1}
\leq 0.
\end{multline}
We are thus in position to state the following local entropy estimate ({\it i.e.} the following discrete analogue of Inequality \eqref{eq:entropy}).

\begin{theorem}[Discrete entropy inequality, implicit upwind scheme]\label{thrm:impl_upw}
Any solution of the scheme \eqref{eq:impl} satisfies, for any $K\in\mesh$ and $0 \leq n \leq N-1$:
\[
\frac{|K|}{\delta t} (\eta_K^{n+1}-\eta_K^n)
+ \sum_{\edge\in \edges(K)} |\edge|\ \eta_\edge^{n+1} u_{K,\edge}^{n+1} \leq 0,
\]
with $\eta_K^m= \varphi_\rho(\rho_K^m) + \rho_K^m\, \varphi_e(e_K^m)$, $m=n,\ n+1$, and $\eta_\edge^{n+1}= \varphi_\rho(\rho_\edge^{n+1}) + \rho_\edge^{n+1} \,\varphi_e(e_\edge^{n+1})$.
\end{theorem}

\begin{proof}
The desired relation is obtained by summing the inequalities \eqref{eq:ent_mi} and \eqref{eq:ent_ei}, using \eqref{prop-entropie}.
\end{proof}
%
Of course, this local entropy inequality also yields the global discrete inequality analogue to \eqref{eq:global-entropy}; furthermore, passing to the limit on the upwind implicit (or pressure correction) scheme applied the Euler equations, this local estimate also yields the Lax consistency property stated in the introduction.
%
\subsection{Beyond the upwind approximation: reducing the diffusion}

The aim of this section is to try to relax the requirement of an upwind approximation for the convection fluxes in \eqref{eq:mass_i} and \eqref{eq:e_int_i}, in order to reduce the numerical diffusion.
We shall see that this leads to a condition which is reminiscent of the limitation requirement which is the core of a MUSCL procedure \cite{van-79-tow}: in order to yield an entropy inequality (instead of, for a MUSCL technique, to yield a maximum principle), the approximation of the unknowns at the face must be "sufficiently close to" the upwind approximation.
The entropy inequality is then obtained only in the weak sense.
The technique to reach this result consists in splitting the rest terms appearing in Lemma \ref{lem:mass} and \ref{lem:energy} in two parts: the first one is non-negative under condition for the face approximation (hence the above mentioned limitation requirement); the second one is conservative, which allows to bound it in a discrete negative Sobolev norm (this explains why the entropy estimate is only a weak one).
This construction relies on the following technical lemma \cite[Lemma 2.3]{gal-08-unc}.

\begin{lemma}\label{lem:int_convexity}
Let $\varphi$ be a strictly convex and continuously differentiable function over $I \subset \xR$.
Let $x_K \in I$ and $x_L \in I$ be two real numbers.
Then the relation
\begin{align}
	\label{eq:conv_int}
	\varphi(x_K) + \varphi'(x_K)\,(x_{KL}-x_K) = \varphi(x_L) + \varphi'(x_L)\,(x_{KL}-x_L ) \mbox{ if } x_K \neq x_L, \\
	x_{KL} = x_K=X_L \mbox{ otherwise }	
\end{align}
uniquely defines the real number $x_{KL}$ in $\li x_K, x_L \ri$.  
\end{lemma}

\begin{remark}[$x_{KL}$ for $\varphi(z)=z^2$]\label{rmrk:xKL}
Let us consider the specific function $\varphi(z)=z^2$.
Then, an easy computation yields 
$x_{KL}=(x_K + x_L) / 2$ \textit{i.e.} the centered approximation.
icici
\end{remark}

\medskip
Let $\varphi_\rho$ be the function defined by \eqref{eq:vphis}
For $\edge\in\edgesint$, $\edge=K|L$, let $\rho_{KL}^{n+1}$ be the real number defined by Equation \eqref{eq:conv_int} with $\varphi = \varphi_\rho$, $x_K=\rho_K^{n+1}$ and $x_L=\rho_L^{n+1}$, and let $(\delta \varphi_\rho)_\edge^{n+1}$ be the following quantity:
\begin{multline}\label{eq:def_drho}
	(\delta \varphi_\rho)^{n+1}_\edge=\varphi_\rho(\rho_K^{n+1})-\varphi_\rho(\rho_\edge^{n+1})
	+ \varphi'_\rho(\rho_K^{n+1})\ \bigr[\rho_{KL}^{n+1}-\rho_K^{n+1}\bigl]
	\\+ \frac 1 2\,\bigl[\varphi'_\rho(\rho_K^{n+1})+\varphi'_\rho(\rho_L^{n+1})\bigr]\ \bigl[\rho_\edge^{n+1}-\rho_{KL}^{n+1}\bigr].
\end{multline}
Note that, since
\[
\varphi_\rho(\rho_K^{n+1})+\varphi'_\rho(\rho_K^{n+1})\ \bigr[\rho_{KL}^{n+1}-\rho_K^{n+1}\bigl]
= \varphi_\rho(\rho_L^{n+1})+\varphi'_\rho(\rho_L^{n+1})\ \bigr[\rho_{KL}^{n+1}-\rho_L^{n+1}\bigl],
\]
the quantity $(\delta \varphi_\rho)^{n+1}_\edge$ only depends on $\edge$.
An easy computation shows that the term associated to the face $\edge$ in the expression \eqref{eq:R_mass} of the remainder term $(R_m)_K^{n+1}$ satisfies:
\begin{align*}
(F_m)_{K,\edge}^{n+1} &=
|\edge|\,\Bigl[\varphi_\rho(\rho_K^{n+1}) - \varphi_\rho(\rho_\edge^{n+1}) + \varphi_\rho'(\rho_K^{n+1}) (\rho_\edge^{n+1}-\rho_K^{n+1})\Bigr] u_{K,\edge}^{n+1}
\\
& = |\edge|\,(\delta \varphi_\rho)^{n+1}_\edge\ u_{K,\edge}^{n+1} + (F^R_m)_{K,\edge}^{n+1}
\end{align*}
with $\displaystyle
(F^R_m)_{K,\edge}^{n+1} = |\edge|\,\frac 1 2\,\Bigl[\varphi'_\rho(\rho_K^{n+1})-\varphi'_\rho(\rho_L^{n+1})\Bigr]
\ (\rho_\edge^{n+1}-\rho_{KL}^{n+1})\ u_{K,\edge}^{n+1}.$

Let us assume that, for $\edge \in \edgesint$, $\edge=K|L$ and for $0 \leq n \leq N-1$:
\begin{equation}\label{eq:H-rho-imp}
(H_\rho^{\rm imp}) \qquad \rho_\edge^{n+1} \in \li \rho_K^{n+1}, \ \rho_{KL}^{n+1} \ri \mbox{ if } u_{K,\edge}^{n+1} \geq 0,
\ \rho_\edge^{n+1} \in \li \rho_L^{n+1},\ \rho_{KL}^{n+1} \ri \mbox{ otherwise}.
\end{equation}
Then, since $\varphi'_\rho$ is an increasing function, $(F^R_m)_{K,\edge}^{n+1}\geq 0$.
Let us define $(\delta\! R_m)_K^{n+1}$, $K \in \mesh$, $0 \leq n \leq N-1$ by:
\begin{equation}
|K|\,(\delta\! R_m)_K^{n+1} = \sum_{\edge\in\edges(K)} |\edge|\ (\delta \varphi_\rho)^{n+1}_\edge\ u_{K,\edge}^{n+1}.
\label{deltaRm}	
\end{equation}
Then, under assumption $(H_\rho^{\rm imp})$, we get:
\begin{multline}\label{eq:ent_mi_c}
\frac{|K|}{\delta t} \bigl[\varphi_\rho(\rho_K^{n+1})-\varphi_\rho(\rho_K^n)\bigr]
+ \sum_{\edge\in \edges(K)} |\edge|\ \varphi_\rho(\rho_\edge^{n+1}) u_{K,\edge}^{n+1}
\\
+ \Bigl[\rho_K^{n+1} \varphi_\rho'(\rho_K^{n+1}) - \varphi_\rho(\rho_K^{n+1}) \Bigr] \sum_{\edge\in \edges(K)} |\edge|\  u_{K,\edge}^{n+1}
+ |K|\ (\delta\! R_m)_K^{n+1}
\leq 0.
\end{multline}
This inequality is an important step in proving that a limit of the discrete solutions satisfy a weak entropy inequality, provided that the remainder term $\delta\!R_m$ vanishes when the space and time discretization step tend to zero, under stability assumptions on the solution of the scheme; we now state this estimate on the remainder.
We first start by some discrete norms.
Let $h_\mesh$ be the space discretization:
\[
h_\mesh=\max_{K\in\mesh} {\rm diam}(K).
\]

\begin{definition}[Discrete BV semi-norm and weak $W^{-1,1}$ norm]
For a family $(z_K^n)_{K\in\mesh, 0 \leq n \leq N}\subset \xR$, let us define the following norms:
\begin{equation}
\begin{array}{l} \displaystyle
\normxbv{z}= \sum_{n=0}^N \delta t \sum_{\edge=K|L \in \edgesint} |\edge|\ |z^n_L-z^n_K|,
 \\ \displaystyle
\norm{z}_{-1,1}  = \sup_{\displaystyle \psi \in \xC^\infty_c([0,T)\times\bar\Omega)} \quad
\frac 1 {\displaystyle \sup_{\bfx\in\Omega,\ t \in (0,T)} \norm{\gradi \psi(\bfx,t)}}
\ \Bigl[\sum_{n=0}^N \delta t \sum_{K \in \mesh}|K|\ z_K^n \psi_K^n\Bigr],
\end{array}
\end{equation}
where $\psi_K^n$ stands for $\psi(\bfx_K,t_n)$, with $\bfx_K$ the mass center of $K$. 
Note that this latter weak norm is the discrete equivalent of the continuous dual norm of $v \in L^1(\Omega)$, defined by
\[
\norm{v}_{(W^{1,+\infty})'} =  \sup_{\displaystyle \psi \in \xC^\infty_c([0,T)\times\bar\Omega)} \quad
\frac 1 {\displaystyle \sup_{\bfx\in\Omega,\ t \in (0,T)} \norm{\gradi \psi(\bfx,t)}}
\ \int_0^T \int_\Omega v \psi \dx \dt.
\]
\end{definition}

Then, with these notations, we may state the following bound for $\delta\!R_m$.

\begin{lemma}\label{lmm:cons_rho}
Let $M>1$ and let us suppose that $\rho_K^n \leq M$, $1/ \rho_K^n \leq M$ and $|u_{K,\edge}| \leq M$, for $K \in \mesh$, $\edge \in \edges(K)$ and $0 \leq n \leq N$.
Let us denote by $|\varphi_\rho'|_\infty$ the maximum value taken by $|\varphi_\rho'|$ over the interval $[1/M,\ M]$ (which, by convexity of $\varphi_\rho$ is equal to either $|\varphi'_\rho(1/M)|$ or $|\varphi'_\rho(M)|$).
Then, under assumption \eqref{eq:H-rho-imp}, the quantity $\delta\! R_m$ defined by \eqref{deltaRm},\eqref{eq:def_drho} satisfies:
\[
\norm{\delta\! R_m}_{-1,1} \leq 3\, M\ |\varphi_\rho'|_\infty\ \normxbv{\rho}\ h_\mesh.
\]
\end{lemma}

\begin{proof}
Let us consider the quantity $(\delta \varphi_\rho)_\edge^{n+1}$ defined by \eqref{eq:def_drho}.
Since both $\rho_\edge^{n+1}$ and $\rho_{KL}^{n+1}$ lie in the interval $\li \rho_K^{n+1},\ \rho_L^{n+1}\ri$, we have, by convexity of $\varphi_\rho$:
\[
|(\delta \varphi_\rho)_\edge^{n+1}| \leq 3\, \max\bigl(|\varphi'_\rho(\rho_K^{n+1})|,\ |\varphi'_\rho(\rho_L^{n+1})|\bigr)\ |\rho_K^{n+1}-\rho_L^{n+1}|.
\]
Let $\psi$ be a function of $\xC^\infty_c(\Omega\times(0,T))$.
We have, thanks to the conservativity of the remainder term:
\begin{align*}
T&=
\sum_{n=0}^{N-1} \delta t \sum_{K \in \mesh}|K|\ (\delta\! R_m)_K^{n+1} \psi_K^{n+1} \\
&=
\sum_{n=0}^{N-1} \delta t \sum_{\edge=K|L \in \edgesint} |\edge|\ (\delta \varphi_\rho)_\edge^{n+1}\ (\psi_K^{n+1}-\psi_L^{n+1})\ u_{K,\edge}.
\end{align*}
Therefore,
\[
|T| \leq 3\, |\varphi'_\rho|_\infty\, M\ \bigl[\max_{\bfx\in\Omega,\ t \in (0,T)} \norm{\gradi \psi(\bfx,t)}\bigr]\ h_\mesh
\sum_{n=0}^{N-1} \delta t \sum_{\edge=K|L \in \edgesint} |\edge|\ |\rho_K^{n+1}-\rho_L^{n+1}|,
\]
which concludes the proof.
\end{proof}

\medskip
Following the same line of thought for the internal energy balance, for $\edge\in\edgesint$, $\edge=K|L$, and $0\leq n \leq N-1$, we denote by $e_{KL}^{n+1}$ the real number defined by Equation \eqref{eq:conv_int} with $\varphi$ being the function $\varphi_e$ defined by \eqref{eq:vphis}, $x_K=e_K$ and $x_L=e_L$.
Then we denote by $(\delta \varphi_e)_\edge^{n+1}$ the following quantity:
\begin{multline}
	\label{eq:def_de}
	(\delta \varphi_e)^{n+1}_\edge=\varphi_e(e_K^{n+1})-\varphi_e(e_\edge^{n+1})
+ \varphi'_e(e_K^{n+1})\ \bigr[e_{KL}^{n+1}-e_K^{n+1}\bigl]
\\+ \frac 1 2\,\bigl[\varphi'_e(e_K^{n+1})+\varphi'_e(e_L^{n+1})\bigr]\ \bigl[e_\edge^{n+1}-e_{KL}^{n+1}\bigr],
\end{multline}
and by $(\delta \! R_e)_K^{n+1}$ the remainder term given by:
\begin{equation}
	|K|\,(\delta \!R_e)_K^{n+1} = \sum_{\edge\in\edges(K)}  (\delta \varphi_e)^{n+1}_\edge\ F_{K,\edge}^{n+1}
	\label{deltaRe}
\end{equation}

Next we introduce the assumption:
\begin{equation}\label{eq:He}
(H_e^{\rm imp}) \qquad 
\left|\begin{array}{l} 
e_\edge^{n+1}\in \li e_K^{n+1},\ e_{KL}^{n+1}\ri \mbox{ if } u_{K,\edge}^{n+1} \geq 0, \\
 e_\edge^{n+1}\in \li e_L^{n+1},\ e_{KL}^{n+1}\ri \mbox{ otherwise}.
	\end{array}\right.
\end{equation}
Then, under assumption $(H_e^{\rm imp})$, we get:
\begin{multline}\label{eq:ent_ei_c}  \frac{|K|}{\delta t} \Bigl[\rho_K^{n+1} \varphi_e(e_K^{n+1})-\rho_K^n \varphi_e(e_K^n)\Bigr]
+ \sum_{\edge\in \edges(K)} \varphi_e(e_\edge^{n+1}) F_{K,\edge}^{n+1}
\\
+ \varphi_e'(e_K^{n+1}) p_K^{n+1} \sum_{\edge\in \edges(K)} |\edge|\  u_{K,\edge}^{n+1}
+ |K|\ (\delta\!R_e)_K^{n+1}
\leq 0.
 \hspace{10ex} \end{multline}
In addition, $\delta\!R_e$ satisfies the following estimate.

\begin{lemma}\label{lmm:cons_e}
Let $M$ be a real number greater than 1 and let us suppose that $\rho_K^n \leq M$, $e_K^n \leq M$, $1/ e_K^n \leq M$ and $|u_{K,\edge}| \leq M$, for $K \in \mesh$, $\edge \in \edges(K)$ and $0 \leq n \leq N$.
Let us define $|\varphi_e'|_\infty=\max(|\varphi'_e(1/M)|,\ |\varphi'_e(M)|)$.
Then, under assumption \eqref{eq:He}, the quantity $\delta\! R_e$ defined by \eqref{deltaRe},\eqref{eq:def_de} satisfies:
\[
\norm{\delta\! R_e}_{-1,1}  \leq 3\, M^2\ |\varphi_e'|_\infty\ \normxbv{e}\ h_\mesh.
\]
\end{lemma}

\medskip
Combining the inequalities \eqref{eq:ent_mi_c} and \eqref{eq:ent_ei_c} and thanks to \eqref{prop-entropie},  Lemma \ref{lmm:cons_rho} and Lemma \ref{lmm:cons_e}, we obtain the following result.

\begin{theorem}[Global and weak local entropy inequalities, implicit reduced diffusion scheme]\label{thrm:impl_muscl}
Under assumptions $(H_\rho^{\rm imp})$ and $(H_e^{\rm imp})$, any solution of the scheme \eqref{eq:impl} satisfies, for any $K\in\mesh$ and $0 \leq n \leq N-1$:
\[
\frac{|K|}{\delta t} (\eta_K^{n+1}-\eta_K^n)
+ \sum_{\edge\in \edges(K)} |\edge|\ \eta_\edge^{n+1} u_{K,\edge}^{n+1} + |K|\ (\delta\!R_\eta)_K^{n+1} \leq 0,
\]
where the remainder term $\delta\!R_\eta=\delta\!R_m + \delta\!R_e$ is conservative, so integrating in space ({\it i.e.} summing over the cells) yields the following global entropy estimate, for $0 \leq n \leq N-1$:
\[
\sum_{K\in\mesh} |K|\ \eta_K^{n+1} \leq \sum_{K\in\mesh} |K|\ \eta_K^n.
\]
In addition, let us suppose that $\rho_K^n \leq M$, $1/ \rho_K^n \leq M$, $e_K^n \leq M$, $1/e_K^n \leq M$ and $|u_{K,\edge}| \leq M$ for $K \in \mesh$, $\edge \in \edges(K)$ and $0 \leq n \leq N$, and let us define the quantities $|\varphi_\rho'|_\infty=\max(|\varphi'_\rho(1/M)|,\ |\varphi'_\rho(M)|)$ and $|\varphi_e'|_\infty=\max(|\varphi'_e(1/M)|,\ |\varphi'_e(M)|)$.
Then the remainder term satisfies the following bound:
\[
\norm{\delta \! R_m}_{-1,1,\star}  \leq 3\ M\ \bigl(|\varphi_\rho'|_\infty\ \normxbv{\rho}+M\ |\varphi_e'|_\infty\ \normxbv{e})\ h_\mesh.
\]
\end{theorem}

The entropy inequality is thus satisfied in the global sense and in the local weak sense for the proposed reduced diffusion adaptation of the implicit upwind scheme.
As mentioned in the introduction of this section, conditions $(H_\rho^{\rm imp})$ and $(H_e^{\rm imp})$ may be seen as an additional constraint to be added to the limitation of a MUSCL-like procedure (see also the conclusion of the last section of this paper).
%
%
\section{Explicit schemes} \label{sec:explicit}

The general form of the discrete analogue of System \eqref{eq:cont} for an explicit scheme reads:
\begin{subequations}\label{eq:expl}
\begin{align}
\nonumber
& \mbox{For } K \in \mesh,\ 0 \leq n \leq N-1, 
\\[1ex] 
\label{eq:mass_e} &
\frac{|K|}{\delta t} (\rho_K^{n+1}-\rho_K^n) + \sum_{\edge\in \edges(K)} F_{K,\edge}^n = 0,
\\[1ex] 
\label{eq:e_int_e} &
\frac{|K|}{\delta t} (\rho_K^{n+1}e_K^{n+1}-\rho_K^n e_K^n) + \sum_{\edge\in \edges(K)} F_{K,\edge}^n e_\edge^n
+ p_K^n \sum_{\edge\in \edges(K)} |\edge|\,u_{K,\edge}^n \geq 0,
\\[1ex] \label{eq:etat_e} &
p_K^n=(\gamma-1)\, \rho_K^n\, e_K^n,
\end{align}\end{subequations}
where the numerical mass flux $F_{K,\edge}^n$ is still defined by \eqref{mass-flux}.
%
%
\subsection{Discrete renormalized forms of the mass balance equation}\label{subsec:rho_exp}

The aim of this section is to derive a discrete analogue of Relation \eqref{eq:ent_m}.
Let $\varphi$ be a twice continuously differentiable convex function from $(0,+\infty)$ to $\xR$, and, mimicking the formal computation performed at the continuous level,  let us multiply \eqref{eq:mass_e} by $\varphi'(\rho_K^{n+1})$.
We get:
\[
\varphi'(\rho_K^{n+1}) \Bigl[\frac{|K|}{\delta t} (\rho_K^{n+1}-\rho_K^n)+ \sum_{\edge\in \edges(K)} F_{K,\edge}^n \Bigr] =
(T_1)_K^{n+1}+(T_2)_K^{n+1}+ |K|\,R_K^{n+1}=0,
\]
with
\begin{equation}\label{eq:def_ren}
\begin{array}{l}\displaystyle
(T_1)_K^{n+1}=\varphi'(\rho_K^{n+1})\ \Bigl[\frac{|K|}{\delta t} (\rho_K^{n+1}-\rho_K^n)\Bigr],
\\[3ex] \displaystyle
(T_2)_K^{n+1}= \varphi'(\rho_K^n)\ \bigl[\sum_{\edge\in \edges(K)} F_{K,\edge}^n\bigr],
\\[4ex] \displaystyle
|K|\,R_K^{n+1}=\bigl(\varphi'(\rho_K^{n+1})-\varphi'(\rho_K^n)\bigr)\ \bigl[\sum_{\edge\in \edges(K)} F_{K,\edge}^n\bigr].
\end{array}
\end{equation}
By a Taylor expansion, we obtain for the first term that there exists $\rho_K^{n+1/2} \in \li\rho_K^n,\ \rho_K^{n+1}\ri$ such that:
\begin{multline}\label{eq:R_1}
(T_1)_K^{n+1}= \frac{|K|}{\delta t} \bigl[\varphi(\rho_K^{n+1})-\varphi(\rho_K^n)\bigr]+|K|\,(R_1)_K^{n+1},\quad
\\ \mbox{with } (R_1)_K^{n+1}= \frac 1 {2 \delta t} \,\varphi''(\rho_K^{n+1/2})\,(\rho_K^{n+1}-\rho_K^n)^2 \geq 0.
\end{multline}
The term $(T_2)_K^{n+1}$ reads:
\[
(T_2)_K^{n+1} = \sum_{\edge\in \edges(K)} |\edge|\, \varphi(\rho_\edge^n) u_{K,\edge}^n
+ \bigl(\varphi'(\rho_K^n) \rho_K^n - \varphi(\rho_K^n)\bigr)\ \bigl[\!\sum_{\edge\in \edges(K)} \! \! |\edge|  u_{K,\edge}^n\bigr]
+ |K|\,(R_2)_K^{n+1},
\]
with
\[
|K|\,(R_2)_K^{n+1} =  \sum_{\edge\in \edges(K)} |\edge| 
\Bigl[\varphi(\rho_K^n) + \varphi'(\rho_K^n) (\rho_\edge^n - \rho_K^n) - \varphi(\rho_\edge^n) \Bigr] u_{K,\edge}^n.
\]
As for the implicit case, let us define $\rho_{KL}^n$ the real number defined by Equation \eqref{eq:conv_int} (and denoted in this relation by $x_{KL}$) with $x_K=\rho^n_K$ and $x_L=\rho^n_L$ and let us assume that, for $\edge \in \edgesint$, $\edge=K|L$ and for $0 \leq n \leq N-1$:
\begin{equation}\label{eq:H_rho^exp}
(H_\rho^{\rm exp}) \qquad \rho_\edge^n \in \li \rho_K^n,\ \rho_{KL}^n \ri \mbox{ if } u_{K,\edge}^n \geq 0,
\ \rho_\edge^n \in \li \rho_L^n,\ \rho_{KL}^n\ri \mbox{ otherwise.}
\end{equation}
Then, under assumption $(H_\rho^{\rm exp})$, the remainder $R_2$ is a sum of a non-negative part and a term tending to zero, as stated in the following lemma.
The proof of this result is very close to the implicit case and is not reproduced here (indeed, up to a change of time exponents at the right-hand side from $n$ to $n+1$, the expression of $(R_2)_K^{n+1}$ is the same that the expression \eqref{eq:R_mass} of $(R_m)_K^{n+1}$, and the computation from Relation \eqref{eq:def_drho} up to the end of the proof of Lemma \ref{lmm:cons_rho} may be reproduced, still with the same change of time exponents).

\begin{lemma}
Let $M > 1$  and let us suppose that $\rho_K^n \leq M$, $1/ \rho_K^n \leq M$ and $|u_{K,\edge}| \leq M$, for $K \in \mesh$, $\edge \in \edges(K)$ and $0 \leq n \leq N$.
Let us define $|\varphi'|_\infty=\max(|\varphi'(1/M)|,\ |\varphi'(M)|)$.
Then there exists $\delta\! R_2$ such that:
\[
R_2 \geq \delta\! R_2 \mbox{ and }
\norm{\delta\! R_2}_{-1,1,\star}  \leq 3M\ |\varphi'|_\infty\ \normxbv{\rho}\ h_\mesh.
\]
\end{lemma}

\medskip
Finally, let $\underline h_\mesh$ be defined by:
\begin{equation}\label{eq:def_h_bar}
\underline h_\mesh = \min_{K \in \mesh} \frac{|K|}{\displaystyle \sum_{\edge \in \edges(K)} |\edge|}.
\end{equation}
Then the term $R$ defined by Equation \eqref{eq:def_ren} satisfies the estimate stated in Lemma \ref{lem:remainder-ex} below, which uses the time BV semi-norm.

\begin{definition}[Discrete time BV semi-norm]\label{def:time-BV}
For a family $(z_K^n)_{K\in\mesh, 0 \leq n \leq N}\subset \xR$,  the time BV semi norm is defined by 
\begin{equation}
\normtbv{z}= \sum_{n=0}^N\ \sum_{K \in \mesh} |K|\ |z^{n+1}_K-z^n_K|,
\end{equation}
\end{definition}

\begin{lemma} \label{lem:remainder-ex}
Let $M > 1$ and let us suppose that $\rho_K^n \leq M$, $1/ \rho_K^n \leq M$ and $|u_{K,\edge}| \leq M$, for $K \in \mesh$, $\edge \in \edges(K)$ and $0 \leq n \leq N$.
Let us denote by $|\varphi''|_\infty$ the maximum value taken by $\varphi''$ on the interval $[1/M,\ M]$.
The remainder term $R$ defined in \eqref{eq:def_ren} satisfies:
\[
\norm{R}_{L^1} =\sum_{n=0}^{N-1} \delta t \sum_{K\in\mesh} |K|\, R_K^n  \leq M^2 \ |\varphi''|_\infty\ \normtbv{\rho}\ \frac{\delta t}{\underline h_\mesh}.
\]
\end{lemma}
\begin{proof}
For $K \in \mesh$ and $0 \leq n \leq N$, we get:
\begin{align}\label{eq:def_R}
|K|\, R_K^{n+1} &=\bigl(\varphi'(\rho_K^{n+1})-\varphi'(\rho_K^n)\bigr)\ \bigl[\sum_{\edge\in \edges(K)} F_{K,\edge}^n\bigr]
\\ &= \varphi''(\tilde \rho_K^{n+1/2}) \bigl(\rho_K^{n+1}-\rho_K^n\bigr)\ \bigl[\sum_{\edge\in \edges(K)} |\edge|\, \rho_\edge^n u_{K,\edge}^n\bigr],
\end{align}
where $\tilde \rho_K^{n+1/2} \in \li \rho_K^n, \rho_K^{n+1}\ri$.
Thus we have:
\[
\norm{R}_{L^1}=\sum_{n=0}^{N-1} \delta t \sum_{K\in\mesh} |K|\, R_K^n \leq
|\varphi''|_\infty\, M^2\ \sum_{n=0}^{N-1} \delta t \Bigl(\sum_{K\in\mesh} |\edge|\Bigr)\ |\rho_K^{n+1}-\rho_K^n|,
\]
which yields the result.
\end{proof}

Let us now suppose that the discretization of the density at the face $ \rho_\edge^n$ is upwind.
Then $R_2$ satisfies:
\begin{equation}\label{eq:R_2_upwind}
|K|\,(R_2)_K^{n+1} =  \sum_{\edge=K|L} \frac 1 2\ |\edge|\ \varphi''(\rho_{K,\edge}^n)\ \bigl(\rho_K^n - \rho_L^n \bigr)^2 (u_{K,\edge}^n)^-,
\end{equation}
where $\rho_{K,\edge}^n \in \li \rho_K^n,\ \rho_L^n \ri$.
Therefore, $R_2$ is non-negative.
Starting from Equation \eqref{eq:def_R}, we may now reformulate the remainder term $R_K^{n+1}$ as $R_K^{n+1}=(R_{01})_K^{n+1} + (R_{02})_K^{n+1}$ with:
\begin{equation}
\begin{array}{l}\displaystyle
|K|\,(R_{01})_K^{n+1}= \varphi''(\tilde \rho_K^{n+1/2}) \bigl(\rho_K^{n+1}-\rho_K^n\bigr)\ \rho_K^n \bigl[\sum_{\edge\in \edges(K)} |\edge|\ u_{K,\edge}^n\bigr],
\\[3ex] \displaystyle
|K|\,(R_{02})_K^{n+1}=\varphi''(\tilde \rho_K^{n+1/2}) \bigl(\rho_K^{n+1}-\rho_K^n\bigr)
\ \bigl[\sum_{\edge\in \edges(K)} |\edge| (\rho_\edge^n-\rho_K^n) u_{K,\edge}^n\bigr].
\end{array}
\end{equation}
By Young's inequality, the second term may be estimated as follows:
\begin{multline*}
|K|\, |(R_{02})_K^{n+1}| \leq  \frac 1 2 \sum_{\edge\in \edges(K)}
|\edge|\ \varphi''(\rho_{K,\edge}^n)\ (u_{K,\edge}^n)^-\ \bigl(\rho_K^n-\rho_L^n\bigr)^2
\\+ \frac 1 2 \sum_{\edge\in \edges(K)}
|\edge|\ \frac{\varphi''(\tilde \rho_K^{n+1/2})^2}{\varphi''(\rho_{K,\edge}^n)}\ (u_{K,\edge}^n)^- \bigl(\rho_K^{n+1}-\rho_K^n\bigr)^2.
\end{multline*}
Therefore, in view of the expressions \eqref{eq:R_1} and \eqref{eq:R_2_upwind} of $(R_1)_K^n$ and $(R_2)_K^n$ respectively, we get $(R_1)_K^{n+1}+(R_2)_K^{n+1}+ (R_{02})_K^{n+1} \geq 0$ under the CFL condition:
\begin{equation}\label{eq:cfl_rho}
\delta t \leq \frac{|K|}{\displaystyle
\sum_{\edge\in \edges(K)} \frac{\varphi''(\tilde \rho_K^{n+1/2})^2}{\varphi''(\rho_{K,\edge}^n)}\ |\edge|\ (u_{K,\edge}^n)^-}.
\end{equation}
To proceed, we now need to suppose that the normal face velocities $u_{K,\edge}$ are computed from a discrete velocity field $\bfu$ by the formula $u_{K,\edge}=\bfu_\edge \cdot \bfn_{K,\edge}$, where $\bfn_{K,\edge}$ is the unit normal vector to $\edge$ outward $K$ and $\bfu_\edge$ is an approximation of the velocity at the face, which may be the discrete unknown itself (in the case of a staggered discretization) or an interpolation (for instance, for a colocated arrangement of the unknowns). 
For $1 \leq q$, let us now define the following discrete norm for a velocity field $\bfu$:
\[
\norm{\bfu}_{L^q(0,T;W^{1,q}_\mesh)}^q = \sum_{i=1}^d\sum_{n=0}^N \delta t \sum_{K\in\mesh}
\ \sum_{(\edge,\edge')\in \edges(K)^2} |K| \left(\frac{u_{\edge,i}^n-u_{\edge',i}^n}{h_K} \right)^q.
\]
It is reasonable to suppose that, under regularity assumptions of the mesh which need to be made precise in view of the space approximation at hand, this norm is equivalent to the standard finite-volume discrete $L^q(0,T; W^{1,q})$ norm \cite{eym-00-fin}; it is indeed true for usual cells (in particular, with a bounded number of faces)  for staggered discretizations and for a convex interpolation of the velocity at the faces for colocated schemes.
Let $C_\mesh$ be the following parameter, measuring the regularity of the mesh:
\begin{equation}\label{eq:def_Cm}
C_\mesh = \max_{K\in \mesh,\ (\edge,\edge')\in \edges(K)^2}\ \frac{(|\edge|+|\edge'|)\,h_K}{|K|}.
\end{equation}
With these notations, we are in position to state the following estimate for $R_{01}$.

\begin{lemma}
Let $M > 1$ and let us suppose that $\rho_K^n \leq M$ and $1/ \rho_K^n \leq M$, for $K \in \mesh$ and $0 \leq n \leq N$.
Let us denote by $|\varphi''|_\infty$ the maximum value taken by $\varphi''$ on the interval $[1/M,\ M]$.
Then the remainder term $R_{01}$ satisfies the following estimate:
\[
\norm{R_{01}}_{L^1} \leq C\ C_\mesh\ M^{(2p-1)/p}\ |\varphi''|_\infty\ \normtbv{\rho}^{1/p}\ \norm{\bfu}_{L^{q}(0,T;W^{1,q}_\mesh)}\ \delta t^{1/p},
\]
where $p \geq 1$, $q\geq 1$ and $\dfrac 1 p + \dfrac 1 {q}=1$ and the positive real number $C$ only depends on the maximal number of faces of the mesh cells.
\end{lemma}
\begin{proof}
Since $\displaystyle \sum_{\edge\in \edges(K)} |\edge|\ \bfn_{K,\edge}=0$, we may write:
\[
|K|\,(R_{01})_K^{n+1}= \varphi''(\tilde \rho_K^{n+1/2}) \bigl(\rho_K^{n+1}-\rho_K^n\bigr)\ \rho_K^n
\bigl[\sum_{\edge\in \edges(K)} |\edge|\ (\bfu_\edge^n -\bfu_K^n) \cdot \bfn_{K,\edge}\bigr],
\]
where $\bfu_K^n$ stands for the mean value of the face velocities $(\bfu_\edge^n)_{\edge \in \edges(K)}$.
We now observe that:
\[
\Bigl| \sum_{\edge\in \edges(K)} |\edge|\ (\bfu_\edge^n -\bfu_K^n) \cdot \bfn_{K,\edge} \Bigr| \leq
2 \sum_{i=1}^d\ \sum_{(\edge,\edge') \in \edges(K)^2} (|\edge|+|\edge'|)\ |u^n_{\edge,i}-u^n_{\edge',i}|.
\]
Therefore, 
\[
|K|\,|(R_{01})_K^{n+1}| \leq 2\ |\varphi''|_\infty M\ \bigl|\rho_K^{n+1}-\rho_K^n\bigr|
\sum_{i=1}^d\ \sum_{(\edge,\edge') \in \edges(K)^2} (|\edge|+|\edge'|)\ |u^n_{\edge,i}-u^n_{\edge',i}|.
\]
We thus have, thanks to a H\"older estimate, for $p \geq 1$, $q \geq 1$ and $\dfrac 1 p + \dfrac 1 {q}=1$:
\begin{align*}
\norm{R_{01}}_{L^1} & = \sum_{n=0}^{N-1} \delta t \sum_{K\in\mesh} |K|\,(R_{01})_K^{n+1}
\\
& \leq  2\ |\varphi''|_\infty M\ \Bigl[ \delta t \sum_{n=0}^{N-1} \sum_{K\in\mesh} |K|\ \bigl|\rho_K^{n+1}-\rho_K^n\bigr|^p
\Bigl(\sum_{i=1}^d\ \sum_{(\edge,\edge') \in \edges(K)^2} 1\Bigr) \Bigr]^{1/p}
\\
& \Bigl[\sum_{n=0}^{N-1} \sum_{K\in\mesh} \sum_{i=1}^d \sum_{(\edge,\edge') \in \edges(K)^2}  
\delta t\ |K| \ \left(\frac{|u^n_{\edge,i}-u^n_{\edge',i}|}{h_K}\right)^{q} \left(\frac{(|\edge|+|\edge'|)\,h_K}{|K|}\right)^{q}
\Bigr]^{1/q}.
\end{align*}
Using $|\rho_K^{n+1}-\rho_K^n |^p \leq (2 M)^{p-1}\ |\rho_K^{n+1}-\rho_K^n |$ yields the result. 
\end{proof}

\medskip
Gathering the results of this section, we obtain the following proposition.

\begin{proposition}
Let $\varphi$ be a twice continuously differentiable convex function from $(0,+\infty)$ to $\xR$, and let $\rho$ satisfy \eqref{eq:mass_e}.
Let $M \geq 1$ and let us suppose that $\rho_K^n \leq M$, $1/ \rho_K^n \leq M$ and $|u_{K,\edge}| \leq M$, for $K \in \mesh$, $\edge \in \edges(K)$ and $0 \leq n \leq N$.
Let $|\varphi'|_\infty=\max(|\varphi'(1/M)|,\ |\varphi'(M)|)$ and $|\varphi''|_\infty$ be the maximum value taken by $\varphi''$ on the interval $[1/M,\ M]$.
Then the following inequality holds:
\begin{multline*}
\frac{|K|}{\delta t} \bigl[\varphi(\rho_K^{n+1})-\varphi(\rho_K^n)\bigr]
+ \sum_{\edge\in \edges(K)} |\edge| \varphi(\rho_\edge^n) u_{K,\edge}^n
\\+ \bigl(\varphi'(\rho_K^n) \rho_K^n - \varphi(\rho_K^n)\bigr)\ \bigl[\sum_{\edge\in \edges(K)} |\edge|\, u_{K,\edge}^n\bigr]
+ |K|\,(R_\rho)_K^{n+1} \leq 0,
\end{multline*}
where the remainder $(R_\rho)_K^{n+1}$ is defined as follows.
\begin{list}{-}{\itemsep=1ex \topsep=1ex \leftmargin=1.cm \labelwidth=0.3cm \labelsep=0.5cm \itemindent=0.cm}
\item If the discretization of the convection term in \eqref{eq:mass_e} satisfies the assumption $(H_\rho^{\rm exp})$, $R_\rho = R_{\rho,1}+R_{\rho,2}$ with:
\begin{align*}
&\norm{R_{\rho,1}}_{-1,1,\star}  \leq 3M\ |\varphi'|_\infty\ \normxbv{\rho}\ h_\mesh, \\
&\norm{R_{\rho,2}}_{L^1} \leq M^2\ |\varphi''|_\infty\ \normtbv{\rho}\ \frac{\delta t}{\underline h_\mesh},
\end{align*} 	
where $\underline h_\mesh$ is defined by \eqref{eq:def_h_bar}.
\item If the discretization of the convection term in \eqref{eq:mass_e} is upwind, under the CFL condition \eqref{eq:cfl_rho}, we also have (with a different expression for $R_\rho$):
\[
\norm{R_\rho}_{L^1} \leq C\ C_\mesh\ M^{(2p-1)/p}\ |\varphi''|_\infty\ \normtbv{\rho}^{1/p}\ \norm{\bfu}_{L^{q}(0,T;W^{1,q}_\mesh)}\ \delta t^{1/p},
\]
where $p \geq 1$, $q\geq 1$ and $\dfrac 1 p + \dfrac 1 {q}=1$, $C_\mesh$ is defined by \eqref{eq:def_Cm} and $C$ only depends on the maximal number of faces of the mesh cells.
\end{list}
\end{proposition}
%
%
\subsection{Inequalities derived from the internal energy balance}\label{subsec:e_exp}

Thanks to the discrete mass balance equation, for any scalar field $z$, a discrete analogue of the expression $\partial_t(\rho z) + \dive(\rho z \bfu)$ may be transformed to a discrete analogue of $\rho \partial_t z + \rho \bfu \cdot \gradi z$.
Of course, these two expressions are (formally) equal at the continuous level; in the fluid mechanics context, the first form is often referred to as the conservative form, while the second one is called the non-conservative form.
Since the second operator is a transport operator, known to preserve the minimum and maximum bounds of the solution, the transformation, in the discrete setting, from the conservative to the non-conservative formulation was first used to derive maximum preservation properties for the discretization \cite{lar-91-how}.
The second form also naturally allows to combine the derivatives as follows:
\[
\varphi'(z) \bigl[ \rho \partial_t z + \rho \bfu \cdot \gradi z \bigr] = \rho \partial_t \bigl(\varphi(z)\bigr)
+ \rho \bfu \cdot \gradi \bigl(\varphi(z)\bigr)
\]
so
\[
\varphi'(z) \bigl[ \partial_t(\rho z) + \dive(\rho z \bfu) \bigr] = \partial_t(\rho \varphi(z)) + \dive\bigl(\rho \varphi(z) \bfu \bigr),
\]
for any regular function of $\varphi$ from $\xR$ to $\xR$.
Note that this computation is closely linked to maximum principle properties (think of $\varphi(z)=(z^-)^2$ for instance).
The object of this section is to mimick this computation at the discrete level, so the first ingredient is the discrete identity corresponding to the equality of the conservative and non-conservative forms of the convection operator, which read:
\begin{multline}
\frac{|K|}{\delta t} (\rho_K^{n+1} z_K^{n+1}-\rho_K^n z_K^n) + \sum_{\edge\in \edges(K)} F_{K,\edge}^n z_\edge^n  =
\frac{|K|}{\delta t} \rho_K^{n+1} (z_K^{n+1}-z_K^n)\\ + \sum_{\edge\in \edges(K)} F_{K,\edge}^n (z_\edge^n-z_K^n),	 \; K \in \mesh, \; 0 \leq n \leq N-1,
\end{multline}\label{eq:cons_noncons}
for any discrete scalar function $z$.
Let now $\varphi$ be a twice continuously differentiable function of $(0,+\infty)$ to $\xR$, and let us multiply the first two terms of the discrete internal energy balance \eqref{eq:e_int_e} by $\varphi'(e_K^{n+1})$.
Switching from the conservative to the non conservative form, we get:
\begin{multline*} \hspace{3ex}
\varphi'(e_K^{n+1})\ \Bigl[\frac{|K|}{\delta t} (\rho_K^{n+1}e_K^{n+1}-\rho_K^n e_K^n) + \sum_{\edge\in \edges(K)} F_{K,\edge}^n e_\edge^n \Bigr]
\\
= \varphi'(e_K^{n+1})\ \Bigl[\frac{|K|}{\delta t} \rho_K^{n+1} (e_K^{n+1}-e_K^n) + \sum_{\edge\in \edges(K)} F_{K,\edge}^n (e_\edge^n-e_K^n) \Bigr].
\hspace{3ex} \end{multline*}
We now write:
\begin{multline*} \hspace{3ex}
\varphi'(e_K^{n+1})\ \Bigl[\frac{|K|}{\delta t} \rho_K^{n+1} (e_K^{n+1}- e_K^n) + \sum_{\edge\in \edges(K)} F_{K,\edge}^n (e_\edge^n-e_K^n) \Bigr]
\\
= \varphi'(e_K^{n+1})\ \frac{|K|}{\delta t} \rho_K^{n+1} (e_K^{n+1} - e_K^n)
+ \varphi'(e_K^n) \Bigl[ \sum_{\edge\in \edges(K)} F_{K,\edge}^n (e_\edge^n - e_K^n) \Bigr]
+ |K|\,R_K^{n+1},
\hspace{3ex} \end{multline*}
with:
\[
|K|\,R_K^{n+1} = \Bigl[ \varphi'(e_K^{n+1}) - \varphi'(e_K^n) \Bigr] \Bigl[ \sum_{\edge\in \edges(K)} F_{K,\edge}^n (e_\edge^n - e_K^n) \Bigr].
\]
As in the previous section, this remainder term satisfies the following estimate.

\begin{lemma}
Let $M \geq 1$ and let us suppose that $\rho_K^n \leq M$, $e_K^n<M$, $1/ e_K^n \leq M$ and $|u_{K,\edge}| \leq M$, for $K \in \mesh$, $\edge \in \edges(K)$ and $0 \leq n \leq N$.
Let us denote by $|\varphi''|_\infty$ the maximum value taken by $\varphi''$ on the interval $[1/M,\ M]$.
Then, the remainder term $R$ satisfies:
\[
\norm{R}_{L^1} \leq M^2\ |\varphi''|_\infty\ \normtbv{e}\ \frac{\delta t}{\underline h_\mesh},
\]
where $\underline h_\mesh$ is defined by \eqref{eq:def_h_bar}.
\end{lemma}

\medskip
Let us now introduce two additional remainder terms as follows:
\begin{multline}
\varphi'(e_K^{n+1})\ \frac{|K|}{\delta t} \rho_K^{n+1} (e_K^{n+1} - e_K^n)
+ \varphi'(e_K^n) \Bigl[ \sum_{\edge\in \edges(K)} F_{K,\edge}^n (e_\edge^n - e_K^n) \Bigr]
\\
=
\frac{|K|}{\delta t} \rho_K^{n+1} \bigl(\varphi(e_K^{n+1})-\varphi(e_K^n) \bigr)
+ \sum_{\edge\in \edges(K)} F_{K,\edge}^n \bigl( \varphi(e_\edge^n) - \varphi(e_K^n)\bigr)
\\+ |K|\,(R_1)_K^{n+1} + |K|\,(R_2)_K^{n+1},
\end{multline}
with:
\[\begin{array}{l}\displaystyle
|K|\,(R_1)_K^{n+1}= \frac{|K|}{\delta t} \rho_K^{n+1}\, \bigl(\varphi(e_K^n) - \varphi(e_K^{n+1}) - \varphi'(e_K^{n+1})( e_K^n-e_K^{n+1})\bigr),
\\[3ex] \displaystyle
|K|\,(R_2)_K^{n+1}= \sum_{\edge\in \edges(K)} F_{K,\edge}^n\, \bigl(\varphi(e_K^n) + \varphi'(e_K^n) (e_\edge^n - e_K^n)- \varphi(e_\edge^n)\bigr).
\end{array}\]
Switching now from the non-conservative formulation to the conservative one yields:
\begin{multline}
\varphi'(e_K^{n+1})\ \frac{|K|}{\delta t} \rho_K^{n+1} (e_K^{n+1} - e_K^n)
+ \varphi'(e_K^n) \Bigl[ \sum_{\edge\in \edges(K)} F_{K,\edge}^n (e_\edge^n - e_K^n) \Bigr]
\\
=
\frac{|K|}{\delta t} \bigl(\rho_K^{n+1}\varphi(e_K^{n+1})-\rho_K^n\, \varphi(e_K^n) \bigr)
+ \sum_{\edge\in \edges(K)} \!\! F_{K,\edge}^n\, \varphi(e_\edge^n)
+ |K|(R_1)_K^{n+1} + |K|(R_2)_K^{n+1}.
\end{multline}
The remainder $(R_1)_K^{n+1}$ may be written:
\begin{equation}\label{eq:R1-e-eup}
(R_1)_K^{n+1} = \frac 1 {2\delta t} \rho_K^{n+1}\, \varphi''(e_K^{n+1/2})\ (e_K^{n+1}-e_K^n)^2,
\end{equation}
where $e_K^{n+1/2} \in \li e_K^n,\ e_K^{n+1} \ri$.
Since $\varphi$ is supposed to be convex, this term is non-negative.
Let $e_{KL}^n$ be the real number defined by Equation \eqref{eq:conv_int} (and denoted in his latter relation by $x_{KL}$) with $x_K=e^n_K$ and $x_L=e^n_L$, and let us assume that, for $\edge \in \edgesint$, $\edge=K|L$ and for $0 \leq n \leq N-1$:
\begin{equation}\label{eq:e^exp}
(H_e^{\rm exp}) \qquad 
e_\edge^n \in \li e_K^n,\ e_{KL}^n \ri \mbox{ if } u_{K,\edge}^n \geq 0,\ e_\edge^n \in \li e_L^n,\ e_{KL}^n\ri \mbox{ otherwise.}
\end{equation}
Then, by a computation similar to the implicit case, the remainder $R_2$ enjoys the following properties.

\begin{lemma}
Let $M \geq 1$ and let us suppose that $\rho_K^n \leq M$, $e_K^n<M$, $1/ e_K^n \leq M$ and $|u_{K,\edge}| \leq M$, for $K \in \mesh$, $\edge \in \edges(K)$ and $0 \leq n \leq N$.
Let us define $|\varphi'|_\infty=\max(|\varphi'(1/M)|,\ |\varphi'(M)|)$.
Then, under the assumption $(H_e^{\rm exp})$, there exists $\delta\! R_2$ such that:
\[
R_2 \geq \delta\! R_2 \mbox{ and} \quad 
\norm{\delta\! R_2}_{L^1} \leq 3 M^2\ |\varphi'|_\infty\, \normxbv{e}\ h_\mesh.
\]
\end{lemma}

\medskip
Let us now suppose that the discretization of the internal energy convection term is upwind.
In this case, we obtain for $(R_2)_K^{n+1}$:
\begin{equation}\label{eq:R2-e-eup}
|K|\,(R_2)_K^{n+1}=\frac 1 2 \sum_{\edge \in \edges(K)} (F_{K,\edge}^n)^-\, \varphi''(e_{K,\edge}^n) (e_K^n-e_L^n)^2,
\end{equation}
where $e_{K,\edge}^n \in \li e_K^n,\ e_L^n \ri$.
The remainder $R_K^{n+1}$ yields in the upwind case:
\[
|K|\, R_K^{n+1} = -\varphi''(\tilde e_K^{n+1/2})\ (e_K^{n+1}-e_K^n)\ \Bigl[ \sum_{\edge\in \edges(K)} (F_{K,\edge}^n)^- (e_L^n - e_K^n) \Bigr],
\]
where $e_K^{n+1/2} \in \li e_K^n,\ e_K^{n+1} \ri$.
So, thanks to the Young inequality:
\begin{multline*}
|K|\,|R_K^{n+1}| \leq \frac 1 2 \sum_{\edge\in \edges(K)} (F_{K,\edge}^n)^- \, \varphi''(e_{K,\edge}^n) (e_L^n - e_K^n)^2
\\+ \frac 1 2\ (e_K^{n+1}-e_K^n)^2 \sum_{\edge\in \edges(K)} (F_{K,\edge}^n)^- \frac{\varphi''(\tilde e_K^{n+1/2})^2}{\varphi''(e_{K,\edge}^n)}.
\end{multline*}
In view of the expressions \eqref{eq:R1-e-eup} and \eqref{eq:R2-e-eup} of $(R_1)_K^{n+1}$ and $(R_2)_K^{n+1}$ respectively, we obtain that $(R_1)_K^{n+1}+(R_2)_K^{n+1}+R_K^{n+1}\geq 0$ under the following CFL condition:
\begin{equation}\label{eq:cfl_e}
\delta t \leq \frac{\varphi''(e_K^{n+1/2})\ |K|\ \rho_K^{n+1}}
{\displaystyle \sum_{\edge\in \edges(K)} \frac{\varphi''(\tilde e_K^{n+1/2})^2}{\varphi''(e_{K,\edge}^n)}\ (F_{K,\edge}^n)^-}.
\end{equation}

Results of this section are gathered in the following proposition.

\begin{proposition}
Let $\varphi$ be a twice continuously convex function from $(0,+\infty)$ to $\xR$.
Then a discrete identity of the following form holds:
\begin{multline}\label{eq:e_id_exp} 
\varphi'(e_K^{n+1})\ \Bigl[\frac{|K|}{\delta t} (\rho_K^{n+1}e_K^{n+1}-\rho_K^n e_K^n) + \sum_{\edge\in \edges(K)} F_{K,\edge}^n e_\edge^n \Bigr]
\\
\leq \frac{|K|}{\delta t} \bigl(\rho_K^{n+1}\varphi(e_K^{n+1})-\rho_K^n\, \varphi(e_K^n) \bigr)
+ \sum_{\edge\in \edges(K)} F_{K,\edge}^n\, \varphi(e_\edge^n)
+ |K|\,(R_e)_K^n,
\hspace{3ex}\end{multline}
where the remainder term $R_e$ may be chosen to enjoy the following properties:
\begin{list}{-}{\itemsep=1ex \topsep=1ex \leftmargin=1.cm \labelwidth=0.3cm \labelsep=0.5cm \itemindent=0.cm}
\item {\em Case 1}:  the approximation of $e_\edge^n$ in \eqref{eq:e_id_exp} satisfies the assumption $(H_e^{\rm exp})$.
Let $M \geq 1$ and let us suppose that $\rho_K^n \leq M$, $e_K^n<M$, $1/ e_K^n \leq M$ and $|u_{K,\edge}| \leq M$, for $K \in \mesh$, $\edge \in \edges(K)$ and $0 \leq n \leq N$.
Let us define $|\varphi'|_\infty=\max(|\varphi'(1/M)|,\ |\varphi'(M)|)$, and let $|\varphi''|_\infty$ be the maximum value taken by $\varphi''$ on the interval $[1/M,\ M]$.
Then, the remainder term $R_e$ satisfies:
\[
\norm{R_e}_{L^1} \leq
3M^2\ |\varphi'|_\infty\, \normxbv{e}\ h_\mesh
+ M^2\ |\varphi''|_\infty\,\normtbv{e}\ \frac{\delta t}{\underline h_\mesh},
\]
where $\underline h_\mesh$ is defined by \eqref{eq:def_h_bar}.
\item {\em Case 2}: the approximation of $e_\edge^n$ in \eqref{eq:e_id_exp} is upwind.
Under the CFL condition \eqref{eq:cfl_e}, $(R_e)_K^{n+1}=0$.
\end{list}
\end{proposition}
%
%
\subsection{Entropy inequalities}

Entropy inequalities are obtained by applying the results of Sections \ref{subsec:rho_exp} and \ref{subsec:e_exp} with $\varphi=\varphi_\rho$ and $\varphi=\varphi_e$ respectively.

\medskip
Let us define by $\rho_{KL}^n$ (resp. $e_{KL}^n$) the real number defined by Equation \eqref{eq:conv_int} with $x_K=\rho^n_K$ (resp. $x_K=e^n_K$) and $x_L=\rho^n_L$ (resp. $x_L=e^n_L$) and $\varphi=\varphi_\rho$ (resp. $\varphi=\varphi_e$) and let us assume that, for $\edge \in \edgesint$, $\edge=K|L$ and for $0 \leq n \leq N-1$:
\begin{equation}\label{eq:H^exp}
\begin{array}{ll}
(H_\rho^{\rm exp})
\qquad &
\rho_\edge^n \in \li \rho_K^n,\ \rho_{KL}^n \ri \mbox{ if } u_{K,\edge}^n \geq 0,\ \rho_\edge^n \in \li \rho_L^n,\ \rho_{KL}^n\ri \mbox{ otherwise,}
\\[2ex]
(H_e^{\rm exp})
\qquad &
e_\edge^n \in \li e_K^n,\ e_{KL}^n \ri \mbox{ if } u_{K,\edge}^n \geq 0,\ e_\edge^n \in \li e_L^n,\ e_{KL}^n\ri \mbox{ otherwise.}
\end{array}
\end{equation}

\begin{theorem}[Discrete entropy inequalities, explicit schemes]
Let $\rho$ and $e$ satisfy the relations of the scheme \eqref{eq:expl}.
Let $M \geq 1$ and let us suppose that $\rho_K^n \leq M$, $1/ \rho_K^n \leq M$, $e_K^n \leq M$, $1/e_K^n \leq M$ and $|u_{K,\edge}| \leq M$, for $K \in \mesh$, $\edge \in \edges(K)$ and $0 \leq n \leq N$.
Let $|\varphi'_\rho|_\infty=\max(|\varphi'_\rho(1/M)|,\ |\varphi'_\rho(M)|)$, $|\varphi'_e|_\infty=\max(|\varphi'_e(1/M)|,\ |\varphi'_e(M)|)$ and
let us denote by $|\varphi''_\rho|_\infty$ and $|\varphi''_e|_\infty$ the maximum value taken by $\varphi''_\rho$ and $\varphi''_e$ respectively on the interval $[1/M,\ M]$.
Let $\eta$ be defined as in Theorem \ref{thrm:impl_upw}.
Then any solution of the scheme \eqref{eq:expl} satisfies, for any $K\in\mesh$ and $0 \leq n \leq N-1$:
\[
\frac{|K|}{\delta t} (\eta_K^{n+1}-\eta_K^n)
+ \sum_{\edge\in \edges(K)} |\edge|\ \eta_\edge^n u_{K,\edge}^n + |K|\ (R_\eta)_K^n \leq 0,
\]
where the remainder term $R_\eta$ enjoys the following properties, depending on the discretization of the convection term:
\begin{list}{-}{\itemsep=1ex \topsep=1ex \leftmargin=1.cm \labelwidth=0.3cm \labelsep=0.5cm \itemindent=0.cm}
\item {\em Case 1}: the discretization of the convection term in \eqref{eq:mass_e} and \eqref{eq:e_int_e} satisfies the assumption $(H_\rho^{\rm exp})$ and $(H_e^{\rm exp})$ respectively.
We have $R_\eta=R_{\eta,1}+R_{\eta,2}$ with:
\begin{align*}
& \hspace{2cm}\norm{R_{\eta,1}}_{-1,1,\star}  \leq 3M\, \Bigl(|\varphi'_\rho|_\infty\,\normxbv{\rho}+M\ |\varphi'_e|_\infty\,\normxbv{e}\Bigr)\ h_\mesh,
\\
&\hspace{2cm}\norm{R_{\eta,2}}_{L^1} \leq M^2 \ \Bigl(|\varphi''_\rho|_\infty\,\normtbv{\rho}+|\varphi''_e|_\infty\,\normtbv{e}\Bigr)
\ \frac{\delta t}{\underline h_\mesh},
\end{align*}
where $\underline h_\mesh$ is defined by \eqref{eq:def_h_bar}.
\item {\em Case 2}: the discretization of the convection term in \eqref{eq:mass_e} and \eqref{eq:e_int_e} is upwind.
Under the CFL conditions \eqref{eq:cfl_rho} and \eqref{eq:cfl_e}, we also have (with a different expression for $R_\eta$):
\begin{equation}\label{eq:R_ent_e}
\norm{R_\eta}_{L^1} \leq C\ C_\mesh\ M^{(2p-1)/p}\ |\varphi''|_\infty\ \normtbv{\rho}^{1/p}\ \norm{\bfu}_{L^{q}(0,T;W^{1,q}_\mesh)}\ \delta t^{1/p}.
\end{equation}
where $p \geq 1$, $q\geq 1$ and $\dfrac 1 p + \dfrac 1 {q}=1$, $C_\mesh$ is defined by \eqref{eq:def_Cm} and $C$ only depends on the number of faces of the mesh cells.
\end{list}
\end{theorem}

\medskip
This result deserves the following comments:
\begin{list}{-}{\itemsep=1ex \topsep=1ex \leftmargin=1.cm \labelwidth=0.3cm \labelsep=0.5cm \itemindent=0.cm}
\item First, in the explicit case, we are not able to prove an entropy inequality, neither local nor global.
\item The convergence to zero with the space and time step of the remainders is obtained, supposing a control of discrete solutions in $L^\infty$ and discrete BV norms, in two cases: first when the ratio $\delta t/\underline h_\mesh$ tends to zero, second when the $L^{q}(0,T;W^{1,q}_\mesh)$ norm of the velocity does not blow-up too quickly with the space step.
To this respect, let us suppose that we implement a stabilization term in the momentum balance equation appearing in certain turbulence models  \cite{berselli2006math,sagaut2006large}, and reading (in a pseudo-continuous setting, for short and to avoid the technicalities associated to the space discretization, which may be, for instance, colocated or staggered), for $1 \leq i \leq d$:
\begin{equation}\label{eq:mom}
\partial_t(\rho u_i) + \dive(\rho u_i \bfu) + \partial_i p - h_\mesh^\alpha \Delta_{q} u_i =0,
\end{equation}
where $\Delta_q u_i$ is such that
\[
\norm{u_i}_{W^{1,q}_\mesh}^{q} \leq C \int_\Omega -\Delta_{q} u_i\ u_i \dx,
\]
where $C$ is independent of $h_\mesh$.
This kind of viscosity term may be found in turbulence models
Multiplying \eqref{eq:mom} by $u_i$ and integrating with respect to space and time yields:
\begin{equation}\label{eq:mom_int}
\int_0^T \int_\Omega -\Delta_{q} u_i\ u_i \dx \dt =
- \int_0^T \int_\Omega \bigl(\partial_t(\rho u_i) + \dive(\rho u_i \bfu) + \partial_i p\bigr)\ u_i \dx \dt,
\end{equation}
In this relation, the right-hand side may be controlled under $L^\infty$ and BV stability assumptions (remember that, at the discrete level, the BV and $W^{1,1}$ norms are the same), and we obtain an estimate of the $\norm{\bfu}_{L^{q}(0,T;W^{1,q}_\mesh)}$ which may be used in \eqref{eq:R_ent_e}.
A standard first order diffusion-like stabilizing term corresponds to $q=2$ and $\alpha=1$, so yields and estimate of the $L^2(0,T;H^1_\mesh)$ norm of the velocity as $1/h_\mesh^{1/2}$ which is just counterbalanced by the term $\delta t^{1/2}$ ($p=2$), supposing that the CFL number is constant; such a stabilization is thus not sufficient to ensure that the remainder term tends to zero.
What is needed is in fact:
\[
\alpha < q-1.
\]
To avoid an over-diffusion in the momentum balance, this inequality suggests to implement a non-linear stabilization with $q>2$ which, in turn, will allow $\alpha >1$.
With such a trick, we will be able to obtain the desired "Lax-convergence" result: the limit of a convergent sequence of solutions, bounded in $L^\infty$ and BV norms, and obtained with space and time steps tending to zero, satisfies a weak entropy inequality.
\item We introduced in \cite{pia-13-for} a limitation process for a MUSCL-like algorithm for the transport equation, which consists in deriving an admissible interval for the approximation of the unknowns at the mesh faces, in convection terms, thanks to extrema preservation arguments.
This limitation process has been extended to Euler equations in \cite{gas-17-mus}.
The conditions $(H_\rho^{\rm exp})$ and $(H_e^{\rm exp})$ may easily be incorporated in this limitation: indeed, they define also an admissible interval, which is not disjoint from the initial one, since the upwind value belongs to both.
A similar idea (namely restricting the choice for the face approximation in order to obtain an entropy inequality) may be found in \cite{ber-14-ent}.
\end{list}
%
%
\bibliographystyle{abbrv}
\bibliography{./entropy}\end{document}